\documentclass[11pt,leqno]{article}

\usepackage[all]{xy}
\usepackage{amssymb,amsmath,amsthm,  url,rotating}
\usepackage{mathrsfs}
 \usepackage{graphicx,epsfig}
 \usepackage{xcolor,graphicx}
  \usepackage{makeidx}
\newcommand{\n}{\noindent}

\newcommand{\bb}[1]{\mathbb{#1}}
\newcommand{\cl}[1]{\mathcal{#1}}

\newcommand{\ovl}{\overline}

\theoremstyle{plain}
\newtheorem{thm}{Theorem}[section]
\newtheorem{lem}[thm]{Lemma}

\newtheorem{pro}[thm]{Proposition}

\newtheorem{cor}[thm]{Corollary}

\theoremstyle{definition}

\newtheorem{dfn}[thm]{Definition}

\theoremstyle{remark}
\newtheorem{rem}[thm]{Remark}

\def\tilde{\widetilde}

\renewcommand{\tilde}{\widetilde}

\def\C{\bb C}

\def\F{\bb F}

\def\F{\bb F}

\def\phi{\varphi}

\def\n{\noindent}
\def\nl{\nolimits}

\begin{document}

\title{On a Characterization of the Weak Expectation Property (WEP)}

\author{by\\
Gilles Pisier \\
Texas A\&M University\\
College Station, TX 77843, U. S. A.}

 \maketitle
\begin{abstract}  
We give a detailed proof  of a 
new characterization  of the Weak Expectation Property (WEP) announced by Haagerup 
in the 1990's but unavailable (in any form) till now. 
Our main result is motivated by
a well known conjecture of Kirchberg, which is equivalent to the Connes embedding problem. We review the basic relevant facts connecting 
our main theorem with the latter conjecture, along the lines of our forthcoming lecture notes volume on the Connes-Kirchberg problem.
 \end{abstract}
 
 The Weak Expectation Property (WEP), originally introduced by Lance
 \cite{Lan} has drawn a lot of extra attention recently because
 of Kirchberg's work \cite{Kir}
 and in particular his proof that 
 the Connes embedding problem
 is equivalent to the assertion that the $C^*$-algebra
 of the free group  $\F_\infty$ (or $\F_2$) has the WEP (see \S \ref{nuc}).

 This paper is extracted from the draft of our forthcoming lecture notes
 \cite{P5} devoted to tensor products of $C^*$-algebras, and especially
 to the Weak Expectation Property (WEP)
 and the Local Lifting Property (LLP). In particular
  we prove there in full details  the equivalence of the Kirchberg conjecture
  with the Connes embedding problem,  the Tsirelson conjecture and several other conjectures.
  A draft of that book
  can be found on the following URL:
  
  \url{https://www.math.tamu.edu/~pisier/TPCOS.pdf}
  
  (i.e. the author's web page at Texas A\&M University followed by TPCOS.pdf)
  
 While our lecture notes are essentially self-contained,
 the present  text   has been edited 
 for readers already familiar with operator algebra theory by removing
 unnecessary details.
 
In this note we give a 
new characterization of the WEP, announced by Haagerup 
(see Remark \ref{x4}) but unpublished.
Contrary to other similar situations, it seems that no manuscript
  has been circulated. Nevertheless, we suspect that the proof
  below is close to what Haagerup had in mind. In any case
  his previous work from \cite{Ha} plays a crucial role.

 \n  We define the WEP for a $C^*$-algebra $A$
 by the equality    $A\otimes_{\min} {\mathscr{C}}=A\otimes_{\max} {\mathscr{C}}$, where
 ${\mathscr{C}}$ is the full (or maximal) $C^*$-algebra of the free group
 $\F_\infty$. 
 Kirchberg  showed that this property is equivalent to a weak form of extension property
 (a sort of weakening of injectivity), that had been considered  by Lance. 
 More precisely, assuming $A\subset B(H)$, Kirchberg showed that
 $A\otimes_{\min} {\mathscr{C}}=A\otimes_{\max} {\mathscr{C}}$
 holds if and only if
   any $*$-homomorphism
 $u: A \to M$ into a von Neumann algebra $M$
 extends to a contractive c.p. map $\tilde u: B(H) \to M$.
 Let $A\subset B(H)$ be a $C^*$-subalgebra.
  Let $\ovl{A}^\sigma$ denote the weak* closure of $A$ in $B(H)$,
 equal to $A''$ (if $A$ is unital) by the bicommutant theorem. Following Lance \cite{Lan},
 a unital c.p.  mapping $T: B(H) \to \ovl{A}^\sigma$ is called a weak expectation
  \index{weak expectation}
if $T(a)=a$ for any $a\in A$.

  \begin{rem}\label{lanc} 
Lance's original definition of the WEP for a $C^*$-algebra is  different but
       equivalent to ours.   Lance  \cite{Lan} says that a $*$-homomorphism
 $\pi: A \to B(H)$ has the WEP
 if
 the von Neumann algebra it generates,
 i.e.  the weak* closure $\ovl{\pi(A)}^\sigma$,  admits a weak expectation.
 He then says that $A$ has the WEP
 if every \emph{faithful} $\pi$ has the WEP. Concerning the relevance of the latter faithfulness
 assumption, see \cite{Bla1,Bla2,BrDy}.
 \end{rem}

It is convenient
to enlarge Lance's concept, as follows.
\begin{dfn}\label{dgwep}
 Let $A\subset B$ be a $C^*$-subalgebra of another one. A linear  mapping $V: B \to A^{**}$  
will be  called a generalized weak expectation if $\|V\|\le 1$ and
 \index{generalized weak expectation} $V(a)=a$ for any $a\in A$.
 When such a $V$ exists, we will say that the inclusion $A\subset B$
 admits a generalized weak expectation.
  \end{dfn}
\begin{rem}\label{x11} Let $V$ be as above in Definition \ref{dgwep}.
Let $P=\ddot V: B^{**} \to A^{**}$.
Observing that $\ddot V$  is continuous with respect to
$\sigma(B^{**},B^{*})$ and $\sigma(A^{**},A^{*})$, one easily checks that
$P$ is a contractive linear projection onto $A^{**}$. 
\\
Conversely, if there is a contractive projection $P: B^{**} \to A^{**}$,
then $V=P_{|B}: B \to A^{**}$ is a  generalized weak expectation.\\
By  Tomiyama's classical theorem \cite{[To2]}, 
 the contractive projection $P=\ddot V$ is   completely positive and
completely contractive. 
Note $\ddot V_{|B}=V$. Thus any generalized weak expectation
$V: B \to A^{**}$ is automatically completely positive and
completely contractive.
 \end{rem}
Let $(U_j)_{j\ge 1}$ denote the free unitary generators of  $\mathscr{C}$.
For notational convenience, we set $U_0=1$. We then define
\begin{equation}\label{x10}
E_n={\rm span}[U_j\mid 0\le j\le n-1].\end{equation}
One form of our main result is as follows:
\begin{thm}\label{x5}
Let $A\subset B(H)$ be a 
$C^*$-subalgebra.
The following 
 are equivalent:
\begin{itemize}
\item[{\rm (i)}] For any $n\ge 1$ and any $t\in E_n \otimes  A$
we have $$\|  t\|_{ \mathscr{C}\otimes_{\max} A}=\|  t\|_{ \mathscr{C}\otimes_{\min} A}.$$
\item[{\rm (ii)}]  
For any   $t\in \mathscr{C} \otimes  A$
we have $$\|  t\|_{ \mathscr{C}\otimes_{\max} A}=\|  t\|_{ \mathscr{C}\otimes_{\min} A}.$$
In other words (with our definition) $A$ has the WEP.
\end{itemize}
\end{thm}
\begin{rem}\label{x2} The property (i) means that the inclusion
$E_n \otimes_{\min} A \to { \mathscr{C}\otimes_{\max} A}$
is isometric for all $n\ge 1$. By   \cite[Th. 1]{P}
$A$ has the WEP iff  the same inclusion is
\emph{completely} isometric for all $n\ge 1$ (or just for $n=3$), which means that
(ii) still holds when $A \subset  B(H)$ is replaced by $M_N(A)\subset M_N(B(H))$ for any $N\ge 1$.
\end{rem}

We denote by $D(A,B)$ the normed space of decomposable maps
from $A$ to $B$ (see \S \ref{dec} below). We denote by 
$\ell_\infty^n$ the $n$-dimensional commutative $C^*$-algebra
that is just $\C^n$ equipped with the sup-norm and pointwise product.
In terms of decomposable maps, the preceding theorem implies:

\begin{cor} \label{x1} 
Let $i: A\to B(H)$ be the inclusion mapping.
The following 
 are equivalent:
\begin{itemize}
\item[{\rm (i)}] For any $n\ge 1$ and any $T:\ell_\infty^n \to A$
we have $$\|  T\|_{D(\ell_\infty^n ,A)}=\|  T\|_{cb}.$$
\item[{\rm (i)'}] For any $n\ge 1$ and any $T:\ell_\infty^n \to A$
we have $$\|  T\|_{D(\ell_\infty^n ,A)}=\|  iT\|_{D(\ell_\infty^n ,B(H))}.$$
\item[{\rm (ii)}] $A$ has the WEP.
\end{itemize}
\end{cor}

In \cite{JLM}, Junge and Le Merdy proved 
that $A$ has the WEP iff $M_N(A)$ satisfies (i) in Corollary  \ref{x1}
    for all $N\ge 1$. This can be viewed
  as an application of Remark \ref{x2} once one observes that
  by Kirchberg's fundamental theorem (see \cite{P})
 $ { \mathscr{C}\otimes_{\min} B(H)}={ \mathscr{C}\otimes_{\max} B(H)}$
  (isometrically) and for any $T:\ell_\infty^n \to A$ with associated tensor
  $t\in E_n \otimes A$ defined by $t=\sum_1^{n} U_{j-1} \otimes T(e_j)$,
  we have (see \S \ref{x3}) $$\|  T\|_{D(\ell_\infty^n ,A)}=\|t\|_{ \mathscr{C}\otimes_{\max} A}.$$
  Moreover we have (see \S \ref{min})
  $$\|  iT\|_{CB(\ell_\infty^n,B(H))}=\|  T\|_{cb} = \|t\|_{ \mathscr{C}\otimes_{\min} A}.$$
 Therefore (i) in Theorem \ref{x5}
 is the same as (i) in Corollary \ref{x1}.
\\
  The identity $\|  T\|_{cb}=\|  iT\|_{D(\ell_\infty^n,B(H))}$ 
   also shows that in Corollary \ref{x1}   (i)' is but a reformulation of (i).
\begin{rem}\label{x4}
The characterization of the WEP in Corollary \ref{x1}
 was claimed   by Haagerup
 in personal communication to  Junge and Le Merdy while
 they completed their paper \cite{JLM}.
 They do not have   a written trace of the proof.
 Similarly the author, who
 had just written \cite{Pro} and was-at that time-in close contact with Haagerup
  in connection with the latter's related  unpublished manuscript   \cite{H4}
  does not remember being informed about the content of Corollary \ref{x1}.
\end{rem}
 Incidentally the results of the unpublished manuscript   \cite{H4} are now available
 in \cite[chap. 23]{P5}.

Theorem \ref{x5} and Corollary \ref{x1} will be deduced from a more general result
where $B(H)$ is replaced by a general $C^*$-algebra.
Theorem \ref{x5}  and
Corollary \ref{x1} will be proved after Remark \ref{ro18}.

The reader will find detailed proofs of all the ingredients used in the sequel  
in \cite{P5}.

\section{Nuclear pairs}\label{nuc}

We start by a few general remarks around nuclearity for pairs.
  \begin{dfn} A pair of $C^*$ algebras $(A,B)$ will be called a nuclear pair
  \index{nuclear pair} if
  $$A \otimes_{\min} B=A \otimes_{\max} B,$$
  or equivalently if the min- and max-norm are equal on the algebraic tensor product $A\otimes B$.
 \end{dfn}
  \begin{rem}\label{da1} If the min- and max-norm are equivalent on $A\otimes B$,
 then they automatically are equal. 
  \end{rem}
    \begin{rem}\label{da20} 
    Let $A_1\subset A$ and $B_1\subset B$ be $C^*$-subalgebras.
    In general, the nuclearity of the pair $(A,B)$ does \emph{not} imply that of 
    $(A_1,B_1)$. 
    This ``defect" is a major feature of the notion of nuclearity.
    However, if  $(A_1,B_1)$ admit contractive c.p. projections (conditional expectations)
    $P: A \to A_1$ and $Q: B \to B_1$ 
    then  $(A_1,B_1)$ inherits the nuclearity of  $(A,B)$.   \end{rem}
Recall that $A$ is called nuclear
\index{nuclear $C^*$-algebra} if $(A,B)$ is nuclear for all $B$.\\
The basic examples  of nuclear $C^*$-algebras 
include all commutative ones,
the algebra $K(H)$ of all compact operators on an arbitrary Hilbert space $H$, $C^*(G)$
for   all amenable discrete groups $G$ and the Cuntz algebras.

We wish to single out    two fundamental examples
$$\mathscr{B}=B(\ell_2) \quad {\rm and } \quad\mathscr{C}=C^*(\F_\infty).$$
Recall that every separable unital $C^*$-algebra embeds in $\mathscr{B}$ and is a quotient of $\mathscr{C}$. Neithe $\mathscr{B}$ nor $\mathscr{C}$ is nuclear, nevertheless :
\begin{thm}[Kirchberg \cite{Kiuhf}]\label{kir} The pair $(\mathscr{B} ,\mathscr{C})$ is nuclear.
\end{thm}

A simpler proof appears in \cite{P} (or in \cite{P4}, or now in \cite{P5}).

Since Kirchberg \cite{Kir} showed that a $C^*$-algebra $A$
has Lance's WEP iff the pair $(A,\mathscr{C})$ is nuclear,
we took this as our definition of the WEP. Kirchberg \cite{Kir} also showed
that $A$
has a certain local lifting property (LLP) iff the pair $(A,\mathscr{B})$ is nuclear. We again take this as the definition of the LLP.
With this terminology, Theorem \ref{kir} admits the following generalization:
\begin{cor} Let $B,C$ be $C^*$-algebras. If $B$ has   the WEP and $C$ the LLP then the pair
$(B,C)$ is nuclear.
\end{cor}

 In \cite{JP} it was shown that $\mathscr{B}$ failed the LLP,
or equivalently that the pair $(\mathscr{B},\mathscr{B} )$ was not nuclear,
which gave a negative answer to one of Kirchberg's questions in \cite{Kir}.
However, the following major conjecture remains open:\\
{\bf Kirchberg's conjecture :} The pair $(\mathscr{C},\mathscr{C} )$ is nuclear,
or equivalently $\mathscr{C}$ has the WEP.

Kirchberg showed at the end of \cite{Kir} that this conjecture is equivalent
to the Connes embedding problem whether any finite von Neumann algebra
embeds in an ultraproduct of matrix algebras.

The Kirchberg conjecture asserts that the min and max norms
coincide on $  \mathscr{C}\otimes \mathscr{C}   $.
More recently in \cite[Th. 29]{Oz2}, Ozawa   
   proved that to prove the Kirchberg conjecture it suffices to show
   that they coincide on $E_n \otimes E_n$ for all $n\ge 1$,
   where $E_n$ is as in \eqref{x10}.

\section{Biduals}
 
 We will use here the basic facts and   notation on biduals
 of $C^*$-algebras. When $A$ is a $C^*$-algebra and $M$ a von Neumann one, 
  for all $u:A \to M$ we denote
 $$ {\ddot u}=({u^{*}}_{|{M_*}})^*: A^{**} \to M.$$
 The following statement gathers well known facts (the dec-norm and $D(A,M)$ are defined in the next section).
\begin{thm}\label{bidext}
 Let $u:A\to M$ be a linear map from a $C^*$-algebra to a von Neumann algebra.
  \begin{enumerate}
   \item[{\rm (i)}] If $u$ is a $*$-homomorphism then ${\ddot u}:A^{**}\to M$ is a normal $*$-homomorphism.
   \item[{\rm (ii)}] $u\in CP(A,M)\Rightarrow {\ddot u}\in CP(A^{**},M)$ and $\|\ddot{u}\|=\| u\|$.
  \item[{\rm (iii)}] $u\in CB(A,M)\Rightarrow {\ddot u}\in CB(A^{**},M)$ and $\|\ddot{u}\|_{cb}=\| u\|_{cb}$.
   \item[{\rm (iv)}] $u\in D(A,M)\Rightarrow {\ddot u}\in D(A^{**}, M)$ and $\|\ddot{u}\|_{dec}=\| u\|_{dec}$.
  \end{enumerate}
\end{thm}

To put the connection with biduals in proper perspective
we state the following review statement. Part (i) is the celebrated
Choi-Effros theorem based on Connes's work on injective factors,
(ii)  can be found in Kirchberg's \cite{Kir}, and   (iii) is derived from it.
 \begin{thm}\label{kk1}
Let $i_A: A \to A^{**}$ be the natural inclusion, $A$ being a $C^*$-algebra.
\index{$i_A$}
\\
(i) $A$ is nuclear if and only if  for some (or any)
 embedding $A^{**}\subset B(H) $   there is 
a projection $P: B(H) \to A^{**}$ with $\|P\|_{cb}=1$.\\
(ii) $A$ has the WEP if and only if 
 for some (or any)
 embedding  $A \subset B(H) $
there is 
a projection $P: B(H)^{**} \to A^{**}$ with $\|P\|_{cb}=1$.\\
(iii) $A$ is QWEP if and only if 
 for some  
 embedding  $A^{**} \subset B(H)^{**} $
there is 
a projection $P: B(H)^{**} \to A^{**}$ with $\|P\|_{cb}=1$.
\end{thm}

 \begin{rem} \label{trap1} We emphasize that (in sharp contrast with the analogue for injectivity)
 the existence of an embedding $u: A^{**} \subset B(H)^{**}$
 admitting a c.p. contractive projection $P: B(H)^{**} \subset A^{**}$
 does \emph{not} in general imply the WEP for $A$.
 Indeed,  the  QWEP does \emph{not} imply the WEP.
 For instance, for $G=\F_\infty$, the reduced $C^*$-algebra $C_\lambda^*(G)$
 is QWEP, but it is WEP if and only if    $G$ is amenable.    \end{rem}

   \section{Decomposable maps}\label{dec} 
 
 This section is devoted to linear maps that are decomposable
 as  linear combinations of c.p. maps and to  the appropriate norm
 denoted by $\|\cdot\|_{dec}$.
 As will soon be clear, these maps and the dec-norm play the same role for
 the max-tensor product as  cb-maps and the cb-norm do
 with respect to the min-tensor product.
 
 In this section, the letters $A,B,C$ will denote $C^*$-algebras.
 \\  We will denote by $D(A,B)$ the set of all ``decomposable'' maps $u:A\to B$, i.e. the maps that are in the linear span of $CP(A,B)$.  This means that $u\in D(A,B)$ if and only if
there are $u_j \in CP(A,B)\quad(j=1,2,3,4)$ such that
$$u=u_1 -u_2 +i(u_3 -u_4 ).$$

A simple minded choice of norm would be to take $\| u\|=\inf \sum^{4}_{1}\| u_j \|$, but this is not the optimal choice.  In many respects, the ``right'' norm on $D(A,B)$ is the following one, introduced by  Haagerup in \cite{Ha}. \index{Haagerup}
    We denote
 \begin{equation}\label{d11}\| u\|_{dec}=\inf\{\max\{\| S_1\|,\| S_2\|\}\}
 \end{equation}
  where the infimum runs over all maps $S_1 ,S_2\in CP(A,B)$ such that the map
 \begin{equation}\label{d12} V: x\to \left(
\begin{matrix}
 S_1 (x) & u(x)\\
 u(x^* )^* & S_2 (x)
\end{matrix}
\right) \end{equation}
is in $CP(A,M_2 (B))$.  

We will use the notation
$$u_*(x)=u(x^*)^*.$$
Note that $u=u_*$ if and only if $u$ takes self-adjoint elements
of $A$ to self-adjoint elements of $B$.
This holds in particular for any  c.p.  map $u$.

With this notation, we can write
 $$V=\left(\begin{matrix}
  S_1   & u\\
 u_* & S_2 
\end{matrix}\right).$$
Then $D(A,B)$ equipped with the norm $\| \ \|_{dec}$ is a Banach space.

\begin{rem} It is easy to show that the infimum
in the definition \eqref{d11} of the dec-norm is a minimum (i.e. this infimum is attained)
when the range $B$ is a von Neumann algebra, or when
there is a contractive c.p. projection from $B^{**}$ to $B$.
Haagerup  \index{Haagerup} raises in \cite{Ha} the   (apparently still open) question whether
it is always a minimum.
\end{rem}
  
  \begin{lem}\label{d60} The following simple properties hold:
  \begin{itemize}   
  \item[{\rm (i)}] If $u\in CP(A,B)$,  then 
  \begin{equation}\label{x9}
  \| u\|_{dec}=\| u\|_{cb}=\| u\|.\end{equation}
 \item[{\rm (ii)}]   If $u(x)=u(x^* )^*$ (i.e. $u$ is ``self-adjoint'') then
   \begin{equation}\label{d74} \| u\|_{dec}=\inf\{\| u_1 +u_2 \ \|\ \mid\   u_1 ,u_2 \in CP(A,B), \ u=u_1 -u_2 \}.\end{equation}
  \item[{\rm (iii)}] To any  $u: A\to B$ we  associate the self-adjoint mapping $\tilde u=\left(\begin{matrix}
   0 & u\\
 u_* &  0
\end{matrix}\right)$. \\Then $u\in D(A,B) $ if and only if $\tilde u\in D(A,M_2 (B)) $ and $\| u\|_{dec}
=\| \tilde u\|_{dec}$.
\end{itemize}
   \end{lem}
     
 \begin{pro}\label{pdec} The following additional properties hold:
\begin{itemize}
 \item[{\rm (i)}] We have $D(A,B)\subset CB(A,B)$ and
  \begin{equation}\label{ed89}
  \forall u\in D(A,B)\quad \|u\|_{cb}\leq\| u\|_{dec}.\end{equation} 
 
  \item[{\rm (ii)}] If $u\in D(A,B) $ and $ v\in D(B,C)$ then $vu\in D(A,C)$
   and 
   \begin{equation}\label{d45}
   \| vu\|_{dec}\leq\| v\|_{dec}\| u\|_{dec}.
   \end{equation}
\end{itemize}
   \end{pro}

 The preceding results are valid with an arbitrary range.
 However, the special case when the range is $B(H)$ (or is injective)
 is quite important:

\begin{pro}\label{prodec} If $B=B(H)$ or if $B$ is an injective $C^*$-algebra, then
 $$D(A,B)=CB(A,B)$$ and for any $u \in CB(A,B)$ we have
  \begin{equation}\label{d46} \| u\|_{dec}=\| u\|_{cb}.\end{equation} \end{pro}
 
 \begin{lem}[Decomposable maps into a direct sum] \label{a92}
  Let $A$ and  $(B_i)_{i\in I}$
 be   $C^*$-algebras
 and let $B=(\oplus\sum\nl_{i\in I} B_i)_\infty$.
 Let $u:  A \to B$. We denote $u_i=p_i u:  A \to  B_i$.
 Then $u\in D( A , B)$ if 
 only if
 all the $u_i$'s are decomposable 
 with $\sup\nl_{i\in I} \|u_i\|_{dec}<\infty$
 and we have
\begin{equation}\label{a93}\|u\|_{dec}=\sup\nl_{i\in I} \|u_i\|_{dec}.\end{equation}
 \end{lem}
  
  In the von Neumann algebra setting, the next lemma will be useful.
  \begin{lem}[Decomposability extends to the bidual]\label{a276}
 Let $u: A \to M$ be a linear map from a $C^*$-algebra $A$ to a von Neumann algebra $M$.
Then $u\in D(A,M)\Rightarrow {\ddot u}\in D(A^{**}, M)$ and $\|\ddot{u}\|_{dec}=\| u\|_{dec}$.
\end{lem}

  The next statement provides us with simple examples of decomposable maps; actually it can be shown  that  \eqref{de99'} is somewhat optimal, see \eqref{da8} below.
  \begin{pro}\label{da7}
 \begin{itemize}
 \item[{\rm (i)}]   Let
$u\colon \ A\to A$ be defined by $u(x) = a^* xb$ with $a,b\in A$, 
then $$\|u\|_{ dec} \le \|a\|\ \|b\|.$$
   \item[{\rm (ii)}] Let $u\colon \ M_n\to A$ be a linear mapping into
a
$C^*$-algebra. Assume that   $u(e_{ij}) = a_i^* b_j$ with $a_i,b_j$
in $A$.   Let $\|a\|_C=\left\|\sum a^*_ia_i\right\|^{1/2}$. Then  
  \begin{equation}\label{de99} 
\|u\|_{ dec} \le \|a\|_C\ \|b\|_C.\end{equation}
\item[{\rm (iii)}] More generally, if $u(e_{ij}) = \sum_{1\le k\le m} a_{ki}^* b_{kj}$ with $a_{ki},b_{kj}$
in $A$    then
\begin{equation}\label{de99'} 
\|u\|_{ dec} \le \|\sum\nl_{ki} a_{ki}^* a_{ki}\|^{1/2} \ \|\sum\nl_{kj} b_{kj}^* b_{kj}\|^{1/2}.\end{equation}
\end{itemize}
  \end{pro}
\begin{proof} 
(i)  
Let $V\colon \ A\to M_2(A)$ be the mapping defined by
$$V(x) =\left( \begin{matrix} a^*xa&a^*xb\cr b^*xa&b^*xb   \end{matrix}\right).$$
An elementary verification shows that 
$ V(x)  = t^* \left(\begin{matrix}   x&0\cr 0&x\cr  \end{matrix}\right)t   $
where
$t  = 2^{-1/2} \left(\begin{matrix}a&b\\ a&b\\
\end{matrix}\right).
 $
Clearly this shows that $V$ is c.p.  hence by definition of the dec-norm
  we have
$$\|u\|_{ dec} \le \max\{\|V_{11}\|, \|V_{22}\|\}$$
where $V_{11}(x) = a^*xa$ and $V_{22}(x) = b^*xb$. Thus we obtain
$\|u\|_{ dec} \le \max\{\|a\|^2, \|b\|^2\}.$
Applying this to the mapping $x\to u(x) \|a\|^{-1} \|b\|^{-1}$ we find
$\|u\|_{ dec} \le \|a\|\ \|b\|.$  
\\ (ii) Let $a^* = (a^*_1,a^*_2,\ldots, a^*_n), b^* =
(b^*_1,b^*_2,\ldots, b^*_n)$ viewed as  row matrices    with entries in
$A$ (so that $a$ and $b$ are column matrices). 
By homogeneity, it suffices to prove
\eqref{de99} assuming $\|a\|_C=\|a\|_C=1$.
Then, for any $x$ in $M_n$, $u(x)$ can be written as a matrix product:
$$u(x) = a^*xb.$$
We
  again introduce the mapping $V\colon \ M_n\to M_2(A)$ defined by
$V(x) = \left(\begin{matrix} a^*xa&a^*xb \\ b^*xa&  b^*xb  \end{matrix} \right).$
Again we note
$V(x) = t^*\left(\begin{matrix} x&0\\ 0&x  \end{matrix}\right)  t$
where $t = 2^{-1/2}\left(\begin{matrix}a&b\\ a&b  \end{matrix}\right) \in M_{2n \times 2}(A)$ which
shows   that $V$ is c.p.  so we obtain   
$\|u\|_{ dec} \le \max\{\|V_{11}\|, \|V_{22}\|\} \le \max\{\|b\|^2,
\|a\|^2\} = \max \{ \|\sum b^*_jb_j \| ,  \|\sum
a^*_ia_i \| \}$
and by homogeneity this yields \eqref{de99}.
\\ 
(iii) We have $u=\sum u_k$
where $u_k: M_n \to A$ is defined by $u_k(e_{ij})= a_{ki}^* b_{kj}$.
Let $V_k\colon \ M_n\to M_2(A)$ be associated to $u_k$ as in (ii).
Let $\cl V=\sum V_k$. Clearly $\cl V$  is c.p. and hence
$$\|u\|_{ dec} \le \max\{\|\cl V_{11}\|, \|\cl V_{22}\|\}= \max\{\|\cl V_{11}(1)\|, \|\cl V_{22}(1)\|\},$$
which yields \eqref{de99'}, since by homogeneity
   we may asssume  
$\|(a_{ki})\|_C=\|(b_{ki})\|_C=1$.
  \end{proof}

\begin{pro}\label{prodec1} Let $A,B,C$ be $C^*$-algebras.
 For any
 $u\in D(A,B) $    \begin{equation}\label{e89}
  \forall x\in C\otimes A \quad
\|(Id_{C}\otimes u)(x)\|_{C\otimes_{\max} B} \le \|u\|_{dec} \|x\|_{C\otimes_{\max} A}.
\end{equation}
Moreover,   the mapping
$  Id_{C}\otimes u:  C\otimes_{\max} A\to C\otimes_{\max} B  $
is decomposable  and its norm satisfies
 \begin{equation}\label{e89b}\|  Id_{C}\otimes u\|_{D(   C\otimes_{\max} A , C\otimes_{\max} B  )} \le  \|u\|_{dec} .\end{equation}
 A fortiori,
 \begin{equation}\label{x8}\|  Id_{C}\otimes u :   C\otimes_{\max} A  \to C\otimes_{\max} B  \| \le  \|u\|_{dec} .\end{equation}
 \end{pro}

  \begin{cor}\label{cordec2} Let $u_j\in D(A_j,B_j)$ ($j=1,2$)
  be decomposable mappings between $C^*$-algebras.
  Then $u_1\otimes u_2$ extends
  to a decomposable mapping in 
  $D(A_1\otimes_{\max} A_2, B_1\otimes_{\max} B_2)$
  such that
  \begin{equation}\label{eqdec2}
  \|u_1\otimes u_2\|_{   D(A_1\otimes_{\max} A_2, B_1\otimes_{\max} B_2)}
   \le   \|u_1 \|_{dec}   \|u_2\|_{dec}. \end{equation}
  \end{cor}
   
  When the mapping $u$ has finite 
  rank then a stronger result holds. We can go $\min\to \max$:
  \begin{pro}\label{findec}
 Let
 $u\in D(A,B) $ be a  \emph{finite rank} map between $C^*$-algebras.  For
 any $C^*$-algebra $C$
 we have 
  \begin{equation}\label{e89fr}
   \forall x\in C\otimes A \quad
\|(Id_{C}\otimes u)(x)\|_{C\otimes_{\max} B} \le \|u\|_{dec} \|x\|_{C\otimes_{\min} A}.
\end{equation}
 \end{pro}
 \begin{proof} For any finite dimensional subspace $F\subset B$,
 the min and max norms are clearly equivalent on $C\otimes F$.
Thus  since its rank is finite $u$ defines
 a bounded map $Id_{C}\otimes u:  C\otimes_{\min} A\to C\otimes_{\max} B  $.
 That same map has norm at most $ \|u\|_{dec}$
 as a map from $C\otimes_{\max} A$ to $C\otimes_{\max} B  $.
 But since we have a metric surjection
 $q: C\otimes_{\max} A\to C\otimes_{\min} A$ taking the open unit ball onto the open unit ball,
 it follows automatically that
 $$\|Id_{C}\otimes u:  C\otimes_{\min} A\to C\otimes_{\max} B \|= \|Id_{C}\otimes u:  C\otimes_{\max} A\to C\otimes_{\max} B \| \le \|u\|_{dec}.$$ 
 \end{proof}
  
  Essentially all the preceding
  results come from Haagerup's \cite{Ha} where detailed proofs can be found.
   
\section{Dec-norms of mappings versus max-norms of tensors}\label{x3}

\begin{lem}\label{ro27}   Let $\F $ be a free group with (free) generators $(g_i)_{i\in I}$
 and let $U_i=U_{\F}(g_i)\in C^*(\F)$ ($i\in I$).
 We augment $I$ by one element by setting formally  $\dot I= I\cup\{0\}$, 
 and we set $g_0$ equal to the unit in $\F$ so that
 $U_0=U_{\F}(g_0)=1$.
Let
$(x_i)_{i\in \dot I}$ be a finitely supported family in a
$C^*$-algebra $A$ and let $T\colon \ \ell_\infty( {\dot I})\to A$
be the mapping defined by $T((\alpha_i)_{i\in {\dot I}}) =
\sum\nolimits_{i\in {\dot I}} \alpha_i x_i$. Then we have
\begin{equation}\label{impu}
\left\|\sum\nl_{i\in {\dot I}} U_i \otimes
x_i\right\|_{C^*(\F)\otimes _{\rm max} A} = \|T\|_{
dec}. \end{equation}
\end{lem}

 For emphasis, we single out the next example, 
which will play an important role in the sequel.
The reader should compare this to the  description of 
 the  unit ball of $CB(\ell_\infty^n,A)$ in  \eqref{eq81} below.
\begin{lem} Consider a linear mapping $T: \ell_\infty^n \to A$  into a  $C^*$-algebra $A$.
Let $ x_j= T(e_j)$ ($1\le j\le n$).
Then
\begin{equation}\label{da8'}
\|T\|_{dec} = \inf\{\|\sum\nl_{j} a_{j}^* a_{j}\|^{1/2}\|\sum\nl_{j} b_{j}^* b_{j}\|^{1/2} \mid  a_j,b_j \in A,\  x_j=a_j^*b_j\}.\end{equation}
Consider a linear mapping $u: M_n \to A$  into a  $C^*$-algebra $A$.
Let ${\mathfrak a}\in M_n(A)$ be the matrix defined by ${\mathfrak a}_{ij}= u(e_{ij})$.
Then
\begin{equation}\label{da8} \|u\|_{dec} = \inf\{\|\sum\nl_{k,j} a_{kj}^* a_{kj}\|^{1/2}\|\sum\nl_{k,j} b_{kj}^* b_{kj}\|^{1/2} \mid  a,b \in M_n(A),\  {\mathfrak a}=a^*b\}.\end{equation}
\end{lem}

\begin{proof} See \cite[p. 257]{P4}
for the proof of \eqref{da8'}.The proof of \eqref{da8} is similar.
\end{proof}

\section{The analogue of local reflexivity for  decomposable maps}
 
In sharp contrast with c.b. maps, we have
\begin{thm}\label{da50} For any $n$ and any $C^*$-algebra $A$, we have natural
isometric identifications
$$D(M_n, A^{**})= D(M_n, A)^{**}
 \text{  and  }
 D(\ell_\infty^n, A^{**})= D(\ell_\infty^n, A)^{**}.$$
\end{thm}
\begin{proof} Note that the spaces $D(\ell_\infty^n, A^{**})$ (resp. $D(M_n, A^{**})$) and $D(\ell_\infty^n, A)^{**}$
(resp. $ D(M_n, A)^{**}$)
 are setwise identical.
The proof that their norms are equal
 uses   \eqref{da8} and \eqref{da8'}.
\end{proof}
 
\section{CB-norms of mappings versus   min-norms of tensors}\label{min}

The first part of the next result
is based on the classical observation that
a unitary representation
$\pi\colon\ \F \to B(H)$ is entirely determined
by its values $u_i=\pi(g_i)$ on the generators,
and if we let $\pi$ run over all possible
unitary representations, then we obtain
all possible families $(u_i)$ of unitary
operators. The second part is also well known.

\begin{lem}\label{pchoi} 
Let $A\subset B(H)$ be a $C^*$-algebra. Let $\F$ be a free group
\index{free group}
with generators $(g_i)_{i\in I}$. Let $U_i=U_{\F}(g_i)\in C^*(\F)$. 
Let $(x_i)_{i\in
I}$ be a   family in $A$ with only finitely many non-zero terms.
Consider  the linear map $T\colon \
\ell_\infty(I)
\to A$ defined by  
$T((\alpha_i)_{i\in I}) = \sum\nolimits_{i\in I}
\alpha_i x_i$. Then we have
\begin{equation}\label{eq81}\left\|\sum\nl_{i\in I}U_i\otimes
x_i\right\|_{C^*(\F)\otimes_{\rm min} A}=
\|T\|_{cb}=
\sup \{\big\|\sum u_i\otimes x_i\big\|_{\rm
min}\}\end{equation} 
where the sup runs over all possible Hilbert spaces $K$ and all
families
$(u_i)$  of  unitaries on $K$. Actually, the latter supremum
remains the same if we restrict it to 
finite dimensional Hilbert spaces $K$.
Moreover, in the case when $A=B(H)$ with
  $\dim(H)=\infty$,
we have 
\begin{equation}\label{eq82}\left\|\sum\nl_{i\in I}U_i\otimes
x_i\right\|_{C^*(\F)\otimes_{\rm min} B(H)}=
\inf\left\{ \left\|\sum y_iy^*_i\right\|^{1/2}
\left\|\sum z^*_iz_i\right\|^{1/2} 
\right\} \end{equation} where the infimum, which runs
over all possible factorizations $x_i=y_iz_i$
with
$y_i,z_i$ in
$B(H)$,  
is  actually attained.\\
Moreover, all this remains true if we enlarge
the family $(U_i)_{i\in I}$ by including
the unit element of  $C^*(\F)$.
  \end{lem}

\begin{proof} See \cite[p. 155]{P4}.
\end{proof}

 \section{Multiplicative domains}\label{sec-md}

 \n The unreasonable effectiveness of completely positive contractions
 in $C^*$-algebra theory is partially elucidated by the next 
 statement (due to Choi, based on Kadison's earlier work).
 
\begin{thm}\label{md} Let $u\colon \ A\to B$ be a c.p.  map between 
$C^*$-algebras with $\|u\|\le 1$.
\item[{\rm (i)}] Then if $a\in A$ satisfies $u(a^*a) = u(a)^*u(a)$, we have 
necessarily $$
u(xa) = u(x) u(a),{\forall x\in A}$$   and
the set
 of such $a$'s forms an algebra.
\item[{\rm (ii)}]Let $D_u = \{a\in A\mid u(a^*a) = u(a)^* u(a)$ and $u(aa^*) = 
u(a)u(a)^*\}$. Then $D_u$ is a $C^*$-subalgebra of $A$
(called the multiplicative domain of $u$) and $u_{|D_u}$ is a $*$-homomorphism.
Moreover, we have
\begin{equation}\label{a16}
\forall a,b\in D_u \ \forall x\in A\quad 
u(ax)=u(a)u(x),\ u(xb) =   u(x)u(b)\ \text{ and }
u(axb) = u(a) u(x)u(b). \end{equation}
\end{thm}

\section{Module maps in the cyclic case}

\begin{thm}[\cite{Smod}]\label{smith1}
Let $E\subset B(\cl H)$ be an operator space.
Let $u: E \to B(H)$ be a bounded linear map.
Assume that there are  unital
$C^*$-subalgebras  $A_1,A_2 \subset B(\cl H)$ 
and $*$-homomorphisms $\pi_1 :A_1 \to B(H)$
and $\pi_2 :A_2 \to B(H)$
with respect to which $E$ is a bimodule and  $u$ is bimodular,
meaning that for all $a_j\in A_j$ and all $x\in E$ we have
$$a_1xa_2\in E \text{  and  }    u(a_1xa_2)= \pi_1(a_1)u(x)\pi_2(a_2).$$
If $\pi_1$ and $\pi_2$ are cyclic
then $u$ is c.b. and $\|u\|_{cb}=\|u\|   $.
\end{thm}

\section{Reduction to the $\sigma$-finite case}\label{rttsfc}

 We will show that
  if an inclusion of von Neumann algebras $M\subset \cl M$
  satisfies a certain property, say property $\mathscr{P}$,
  then there is a completely contractive projection $P: \cl M \to M$.
  The goal of this (technical)
  section is to show that modulo a simple
  assumption on the property $\mathscr{P}$ we may always restrict to the case when $M$ 
  is $\sigma$-finite.
  The proof will use the following structural theorem.
  
  \begin{thm}[Fundamental reduction to $\sigma$-finite case]\label{S}
 Any von Neumann algebra $M$ admits a decomposition as a direct sum 
$$M\simeq (\oplus\sum\nl_{i\in I} B({\cl H_i}) \bar\otimes  N_i )_\infty$$
where the $N_i$'s are   $\sigma$-finite (=countably decomposable)
and the ${\cl H_i}$'s are   Hilbert spaces.
\end{thm} 
 
See     \cite[Ch. III \S 1 Lemma 7] {Dix}  (p.224 in the French edition
and p. 291 in the English one) for a detailed proof.

 Note that if $M$ is $\sigma$-finite  we can take for $I$ a singleton
 with $N_i=M$ and ${\cl H_i}=\C$. 
 
  The assumptions we wish to make on $\mathscr{P}$ are as follows:
  if $M\subset \cl M$ has property $\mathscr{P}$
  then for any projection $q\in M$ the inclusion
$qMq\subset q{\cl M}q$ (unital with unit $q$)
also has  property $\mathscr{P}$.
Moreover, 
if $\pi : \cl M \to \cl M^1$
is an isomorphism of von Neumann algebras
taking $M$ onto a subalgebra $M^1\subset \cl M^1$,
then we assume that the ``isomorphic inclusion"
$M^1\subset \cl M^1$ also has  property $\mathscr{P}$.

\begin{pro}\label{ro4}
Under the preceding assumptions,
to show that   every inclusion $M\subset \cl M$ with  
property $\mathscr{P}$  admits a completely contractive projection $P: \cl M \to M$,
it suffices to settle the case when $M$ is $\sigma$-finite.
\end{pro}
\begin{proof} Consider a general inclusion $M\subset \cl M$.
By the structural Theorem \ref{S} we may assume
 \begin{equation}\label{be2'}  M = (\oplus\sum\nl_{i\in I}
 B({\cl H_i})  \bar\otimes N_i )_\infty,\end{equation}
 with the $N_i$'s $\sigma$-finite.
 Let $M_i=B({\cl H_i})  \bar\otimes N_i $.
Let $q_i$ be the (central) projection in $M$
 corresponding to  $M_i $ in \eqref{be2'} so that
 $M_i =q_i M=q_i Mq_i$.
 Let $\cl M_i= q_i \cl M q_i$.
By our first assumption on $\mathscr{P}$  the inclusion $ M_i\subset \cl M_i$ satisfies $\mathscr{P}$.
We claim that  we have 
a von Neumann  algebra $\cl N_i 
$, with a subalgebra $  N^1_i \subset \cl N_i$ 
and
an isomorphism $\pi_i: \cl M_i \to B({\cl H_i}) \bar\otimes \cl N_i  $
such that $\pi_i (M_i)= B({\cl H_i}) \bar\otimes   N^1_i  $.
In other words,
the inclusion $ M_i\subset \cl M_i$ is ``isomorphic" in the preceding sense
to the inclusion $B({\cl H_i})  \bar\otimes N^1_i \subset B({\cl H_i})  \bar\otimes \cl N_i$.
Indeed, since $B({\cl H_i})\simeq B({\cl H_i}) \otimes 1 \subset \cl M_i $,   by
  a well known property of 
  the representations of the $B(H)$'s,  for some $\cl N_i$ we have an isomorphism $\pi_i: \cl M_i \to B({\cl H_i}) \bar\otimes \cl N_i  $
  so that 
    $\pi_i(x\otimes 1) = x \otimes 1$  for any $x\in B({\cl H_i})$.
  Then the subalgebra   $1 \otimes N_i \subset M_i \subset \cl M_i$ is mapped by $\pi_i: \cl M_i \to B({\cl H_i}) \bar\otimes \cl N_i  $ 
   to a subalgebra that commutes with 
   $B({\cl H_i}) \otimes 1_{\cl N_i}$, and hence is included in $1 \otimes\cl N_i$.
   Thus we find $N_i^1$ such that $\pi_i( 1 \otimes N_i  )= 1 \otimes N_i^1$, and an isomorphism
   $\psi_i: N_i \to N_i^1$  
   such that $  \pi_i(1\otimes y)= 1 \otimes \psi_i(y)$
   for all $  y\in N_i$.
   It follows that $\pi_i( B({\cl H_i}) \otimes N_i  )= B({\cl H_i}) \otimes N_i^1$,
   and since $\pi_i$ is bicontinuous
   for the weak* topology, we have
   $\pi_i( B({\cl H_i}) \bar\otimes N_i  )= B({\cl H_i}) \bar\otimes N_i^1$.
   This proves the claim.\\
   By our second assumption on $\mathscr{P}$, the inclusion
     $B({  {\cl H_i}})  \bar\otimes N^1_i \subset B({  {\cl H_i}})  \bar\otimes \cl N_i$ 
     satisfies $\mathscr{P}$.
     Let  $r_i$ be a rank 1 projection in $B({ {\cl H_i}})$.
     Let $q_i'= r_i \otimes 1$.
     By our first assumption again, the inclusion
     $q_i'[ B({  {\cl H_i}})  \bar\otimes N^1_i ]q_i' \subset q_i'[ B({  {\cl H_i}})  \bar\otimes \cl N_i ]q_i' $
     (with unit $q_i'$) also     satisfies $\mathscr{P}$.
     The latter being clearly ``isomorphic"
   to the inclusion $N^1_i \subset \cl N_i$
   we conclude that $N^1_i \subset \cl N_i$ satisfies $\mathscr{P}$.
   But now, at last,  since $N^1_i\simeq N_i$ is $\sigma$-finite,
   if we accept the $\sigma$-finite case, we find
   that there is a completely contractive projection $P_i: \cl N_i \to N^1_i$. 
   Therefore, 
   $Id_{B({\cl H_i}) } \otimes P_i$
   defines  a completely contractive projection
   from
   $B({\cl H_i}) \bar\otimes \cl N_i$
   to $B({\cl H_i}) \bar\otimes   N^1_i$. Then 
   $Q_i=\pi_i^{-1}[Id_{B({\cl H_i}) } \otimes P_i] \pi_i $
   is a completely contractive projection
   from
    $\cl M_i$ to $M_i$,
   and hence 
   the mapping $x \mapsto   (Q_i(q_i x q_i) )_{i\in I} \in (\oplus \sum\nl_{i\in I} M_i)_\infty$
   gives us
   a completely contractive projection
   from $\cl M$ onto $M$.
\end{proof}

\section{A  new characterization of generalized  weak expectations and the WEP}\label{ctwwdm}
\index{generalized  weak expectation}

The main result is the following characterization of generalized  weak expectations
(see Definition \ref{dgwep}),
in terms of decomposable maps.
 
\begin{thm}\label{ro5} Let $B$ be a $C^*$-algebra.
Let $i: A\to B$ be the inclusion mapping from
a $C^*$-subalgebra $ A\subset B$.
The following 
 are equivalent:
\begin{itemize}
\item[{\rm (i)}] For any $n\ge 1$ and any $T:\ell_\infty^n \to A$
we have $$\|  T\|_{D(\ell_\infty^n ,A)}=\|  iT\|_{D(\ell_\infty^n ,B)}.$$
\item[{\rm (i)'}] For any $n\ge 1$ and any $t\in E_n \otimes  A$
we have $$\|  t\|_{ \mathscr{C}\otimes_{\max} A}=\|  t\|_{ \mathscr{C}\otimes_{\max} B}.$$
\item[{\rm (ii)}] For any $n\ge 1$ and any $v:\ell_\infty^n \to A^{**}$
we have $$\|  v\|_{D(\ell_\infty^n ,A^{**})}=\|  i^{**} v\|_{D(\ell_\infty^n ,B^{**})}.$$
\item[{\rm (iii)}] There is a completely contractive c.p. projection
$P: B^{**} \to A^{**}$ (in other words by Remark \ref{x11}  the inclusion $i: A \to B$ admits a generalized weak
expectation).
\item[{\rm (iv)}] For any   $C^*$-algebra $C$ and any $t\in  {C} \otimes  A$
we have $\|  t\|_{  {C}\otimes_{\max} A}=\|  t\|_{  {C}\otimes_{\max} B}.$
\item[{\rm (v)}] For any   $t\in \mathscr{C} \otimes  A$
we have $\|  t\|_{ \mathscr{C}\otimes_{\max} A}=\|  t\|_{ \mathscr{C}\otimes_{\max} B}.$
\end{itemize}
\end{thm}
\begin{rem} Curiously, there does not seem to be a direct argument to show (i)' $\Rightarrow$ (iv).
\end{rem}
\begin{rem} 
The equivalences (iii) $\Leftrightarrow $ (iv)   $\Leftrightarrow $ (v)  are due  to Kirchberg \cite{Kir}.
\end{rem}
\begin{proof}[Proof of Theorem \ref{ro5}]
Note that (i) $\Leftrightarrow $ (i)' is immediate by \eqref{impu}.
 We now claim (i) $\Leftrightarrow $ (ii).
This is an immediate consequence
of Theorem \ref{da50}.
Indeed, let $X_n={D(\ell_\infty^n ,A)}$ and $Y_n={D(\ell_\infty^n ,B)}$, viewed as Banach spaces.
Then,   the assertion that $X_n \subset Y_n$ isometrically, which is a reformulation of (i), 
is equivalent to $X^{**}_n \subset Y^{**}_n$ isometrically.  This follows
 from the  classical fact   that a mapping between Banach spaces is isometric if and only if its
 bitranspose is isometric. By Theorem \ref{da50}
we have $X^{**}_n=  D(\ell_\infty^n ,A^{**}) $  and $Y^{**}_n=  D(\ell_\infty^n ,B^{**}) $.
Thus, (i) $\Leftrightarrow $ (ii) follows.
Let us show (iii) $\Rightarrow $ (iv). Assume (iii). Then by \eqref{x8} and \eqref{x9},  for any $C^*$-algebra $C$
we have $$\| Id_C \otimes P: C \otimes_{\max}  B^{**} \to C \otimes_{\max}    A^{**}\|=1,$$
which clearly implies that the natural map $C \otimes_{\max}    A^{**} \to C \otimes_{\max} B^{**}$
is isometric.  Since as is well known (and elementary to check)
the natural morphisms $C \otimes_{\max}    A \subset C \otimes_{\max}    A^{**}$
and $C \otimes_{\max}    B \subset C \otimes_{\max}    B^{**}$
are  isometric, the natural morphism $C \otimes_{\max}    A  \to C \otimes_{\max} B $ also is isometric. This proves (iv) and (iv) $\Rightarrow $ (v) is trivial.
\\ In the converse direction,  (v) $\Rightarrow $ (i)' also is trivial.\\
Thus, to complete the proof, it suffices to prove the remaining equivalence (ii) $\Leftrightarrow $ (iii), which will follow from the next statement about von Neumann algebras applied to the inclusion $A^{**} \subset B^{**}$.
\end{proof}

\begin{thm}\label{ro5'} Let $\cl M$ be a von Neumann algebra.
Let $i: M\to \cl M$ be the inclusion mapping from
a  von Neumann subalgebra $ M\subset \cl M$.
The following 
 are equivalent:
\begin{itemize}
\item[{\rm (i)}] For any $n\ge 1$ and any $T:\ell_\infty^n \to M$
we have $$\|  T\|_{D(\ell_\infty^n ,M)}=\|  iT\|_{D(\ell_\infty^n ,\cl M)}.$$
\item[{\rm (i)'}] For any $n\ge 1$ and any $t\in E_n \otimes M$
we have $$\|  t\|_{ \mathscr{C} \otimes_{\max} M}=\| t \|_{\mathscr{C}\otimes_{\max} \cl M}.$$
\item[{\rm (ii)}] There is a completely contractive c.p. projection
$P: \cl M \to M$.
\item[{\rm (iii)}]  For any $C^*$-algebra $C$  the natural map $C \otimes_{\max}    M \to C \otimes_{\max} \cl M$
is isometric.
\end{itemize}
\end{thm}
\begin{proof}  
We already saw that (i) and (i)' are equivalent by \eqref{impu}.
We first show (i)' $\Rightarrow $ (ii).
We will use the reduction to the $\sigma$-finite case.
Let $\mathscr{P}$ be the property appearing in (i). 
By the results of \S \ref{dec}
it is easy to check that $\mathscr{P}$ satisfies
the assumptions of Theorem \ref{ro4}.
Therefore, to show  (i) $\Rightarrow $ (ii)
we may assume  $M$  $\sigma$-finite.
Then   there is a realization of $M$ in some $B(H)$
such that $M$ has a cyclic vector. Let $M'\subset B(H)$ be the commutant of
$M$ in $B(H)$.
Let $I\subset U(M')\setminus \{1\}$ 
be a set of unitaries in $M'$ that jointly generate $M'$ as a von Neumann algebra, 
and let $\dot I=I\cup \{1\} $.
Let $\F $ be a free group with (free) generators $(g_x)_{x\in I}$.
Let $U_x=U_{\F}(g_x)\in C^*(\F)$ ($x\in I$), set also $U_1 =1_{C^*(\F)}$, and  
let $\sigma: {C^*(\F)} \to M'$ be the unital $*$-homomorphism
defined by $\sigma(U_x)=x$ for all $x\in I$.
  Let     $E= {\rm span}[ U_x \mid x\in {\dot I}] $.
Consider then the linear mapping
$\hat T: E \otimes M  \to B(H) $ defined for any $e\in E,m\in M$ by
$ \hat T( e \otimes m)=\sigma(e) m$ 
 (and extended by linearity to $E \otimes M $).
 Then for any $t\in E \otimes M$
 we have clearly  
 $$\|\hat T (t) \|  \le \|t\|_{C^*(\F) \otimes_{{\max}}   M}.$$ 
Thus (i)' implies
$$\|\hat T : E\otimes_{\ovl{\max}} M \to B(H)\|\le 1$$
where $E\otimes_{\ovl{\max}} M $ is viewed
as a subspace of $C^*(\F) \otimes_{ {\max}} \cl M $ equipped with the induced norm.
\\
By Theorem \ref{smith1} since $M$ has a cyclic vector we have 
$$\|\hat T: E\otimes_{\ovl{\max}} M \to B(H)\|=\|\hat T: E\otimes_{\ovl{\max}} M \to B(H)\|_{cb} .$$
By Arveson's  extension theorem  
there is  $\tilde T: {C^*(\F) \otimes_{{\max}}   \cl M} \to B(H)$
extending $\hat T$ with $\|\tilde T\|_{cb}\le 1$.
$$\xymatrix{&{C^*(\F)\otimes_{{\max}} {\cl M}} \ar@{-->}[dr]^{ \tilde T } \\
& E\otimes_{\ovl{\max}} M \ar@{^{(}->}[u] 
 \ar[r]^{ \hat T \quad } & B(H)}$$
Since $\hat T$ is unital so is $\tilde T$ and hence $\tilde T$ is c.p.
We claim that $E \otimes 1$ (and hence actually $C^*(\F) \otimes 1$) is included in the 
\index{multiplicative domain}
multiplicative domain $D_{\tilde T}$. Indeed,
since $\tilde T (U_x \otimes 1)= \hat T (U_x \otimes 1)=x \in U(M')$ for any $x\in \dot I$, 
we have $U_x \otimes 1 \in D_{\tilde T}$ for any $x\in \dot I$ and 
the claim follows. Let $P: \cl M\to B(H)$ be  defined by
$P(b)= \tilde T ( 1 \otimes b)$.
Then $P$ is  completely contractive and  c.p.   Since 
$U_x \otimes 1 \in D_{\tilde T}$ for any $x\in \dot I$
and since,   
by Theorem \ref{md},
$ \tilde T$ is bimodular with respect to $ D_{\tilde T}$ we have
by a well known argument  (called ``The trick" in \cite{BO} !)  for any $x\subset  \dot I=U(M')$
$$xP(b)=x\tilde T ( 1 \otimes b) =  \tilde T( (U_x \otimes 1) ( 1 \otimes b))=
\tilde T ( U_x \otimes b) =
\tilde T( ( 1 \otimes b) (U_x \otimes 1) )=\tilde T ( 1 \otimes b) x =P(b)x.$$
Since the unitaries in $I$ generate $M'$,
this shows that $P(b)\in (M')'=M$ and completes
the proof that (i)' $\Rightarrow $ (ii).\\
Assume (ii). Then, by \eqref{x8} and \eqref{x9},  for any $C^*$-algebra $C$
we have $$\| Id_C \otimes P: C \otimes_{\max}  \cl M \to C \otimes_{\max}    M\|=1,$$
which clearly implies that the natural map $C \otimes_{\max}    M \to C \otimes_{\max} \cl M$
is isometric. Thus  (ii) $\Rightarrow $ (iii).\\
Lastly, taking $C =\mathscr C$, (iii) $\Rightarrow $ (i)' is trivial. 
\end{proof}

\begin{rem}[The case $n=3$]\label{ro18} 
In the situation of the preceding Theorem \ref{ro5'}
let us merely assume that for
any $T:\ell_\infty^3 \to M$
we have $\|  T\|_{D(\ell_\infty^3 ,M)}=\|  iT\|_{D(\ell_\infty^3 ,\cl M)}.$
If we assume in addition that $M\subset B(H)$ is cyclic and that $M'$
 is generated by a pair of unitaries, then the same proof (now with $\F=\F_2$ and  $|\dot I |=3$) shows
that there is a completely contractive c.p. projection
$P: \cl M \to M$. Thus when $\cl M=B(H)$ we conclude that $M$ is injective.\\
We recall in passing that it is a
longstanding open problem whether any von Neumann algebra
on a separable Hilbert space is generated by a single element or equivalently by
two unitaries. For example,  this single generation problem is open
 for $M_{\F_\infty}$.
Important partial results are known, notably by Carl Pearcy,  see \cite{DoPe} for details
and references.
See Sherman's paper \cite{Sher}  
for the current status of that problem.

 \end{rem}

We now return to the characterization of
the WEP.
 \begin{proof}[Proof of Theorem \ref{x5}]
We apply Theorem \ref{ro5} with $B=B(H)$.
Note that in that case, 
by \eqref{d46} and \eqref{eq81},
we have
$$\|T\|_{D(\ell_\infty^n ,A)}=\|t\|_{\mathscr{C} \otimes_{\max} A}
\text{  and  }\|  iT\|_{D(\ell_\infty^n ,B(H))}=\| T\|_{cb}=\|t\|_{\mathscr{C} \otimes_{\min} A}$$
for any $T: \ell_\infty^n \to A$,
so that (i) in Theorem \ref{x5}
is the same as (i) in Theorem \ref{ro5}.
But when $B=B(H)$, the conditions (namely either (iii) or (v)) in Theorem \ref{ro5} are equivalent to the WEP for $A$.
Indeed,   by (ii)
  in Theorem \ref{kk1}, (iii) in Theorem \ref{ro5} implies the WEP when $B=B(H)$. Alternatively,
  by Kirchberg's Theorem \ref{kir} (or its corollary) we have 
  $\|t\|_{\mathscr{C} \otimes_{\max} B(H)}=\|t\|_{\mathscr{C} \otimes_{\min} B(H)}$ and hence by the injectivity of the min-tensor product
  $$\|t\|_{\mathscr{C} \otimes_{\max} B(H)}=\|t\|_{\mathscr{C} \otimes_{\min} A}.$$
  Thus (v) in Theorem \ref{ro5} implies that
 $ \|t\|_{\mathscr{C} \otimes_{\min} A}=\|t\|_{\mathscr{C} \otimes_{\max} A}$
 for any $t\in {\mathscr{C} \otimes  A}$, which exactly means $A$ has   the WEP.
  The converse implication (ii) $\Rightarrow$ (i) in Theorem \ref{x5} is trivial.
\end{proof}

\begin{proof}[Proof of Corollary \ref{x1}]
We again invoke Theorem \ref{ro5} with $B=B(H)$.\\
By \eqref{impu} we have $\|  T\|_{D(\ell_\infty^n ,A)}=\|t\|_{\mathscr{C} \otimes_{\max} A}$
and by 
\eqref{eq81} and \eqref{d46}
$\|iT\|_{cb}=\|i T\|_{D(\ell_\infty^n ,B(H))}=\|t\|_{\min}$.
Therefore, either (i) or (i)' in Corollary \ref{x1} 
is equivalent to (i) (and hence to (iii) or (v)) in Theorem \ref{ro5}.
Thus, as in the preceding proof, this is equivalent to  the WEP for $A$.\end{proof}


Since WEP and injectivity are equivalent 
for von Neumann algebras, we can now
recover Haagerup's original result (see \cite{Ha}):

\begin{cor}\label{ro10} When $A$ is a von Neumann algebra,
the assertion (i) in Corollary \ref{x1}
holds if and only if $A$ is injective.
\end{cor}
 
 \begin{rem}
  Although we are left guessing what his argument was to prove Theorem \ref{x5},
  the results of \S \ref{ctwwdm} 
  seem very likely to be close to what Haagerup had in mind.
  Note that the question whether Corollary \ref{x1} holds
  is implicit in Haagerup's previous fundamental (published) paper
  \cite{Ha}, where he proves Corollary \ref{ro10}
  and then
    asks explicitly
  whether for a von Neumann algebra $M$ the isometric identity $D(\ell_\infty^3,M)=CB(\ell_\infty^3,M)$
  implies its injectivity.
  In other words he asks whether (i) in Corollary \ref{x1}
  with $n=3$  suffices to imply the same for all $n$. This is still open,
  but it holds if $M\subset B(H)$ is cyclic and  $M'$ generated by a pair of unitaries (see Remark \ref{ro18}).
  As observed by Junge and Le Merdy in \cite{JLM}
  it also holds if the equality $D(\ell_\infty^3,M)=CB(\ell_\infty^3,M)$
  is meant in the completely isometric sense,
  i.e.
  one assumes the same isometric identity for $M_n(M)$ instead of $M$ 
  for all $n\ge 1$.
  This follows from 
  the same  idea (from \cite[Th. 1]{P}) used in Remark \ref{x2}, but applied
  with $C^*(\F_2)$ in place of $\mathscr C$. \end{rem}

\end{document}